\documentclass[a4paper, reqno]{amsart}

\usepackage[T1]{fontenc}
\usepackage{lmodern}
\usepackage{amssymb,amsmath}
\usepackage[utf8]{inputenc}
\usepackage[unicode=true]{hyperref}
\hypersetup{
  breaklinks=true,
  pdfauthor={Anders Claesson},
  pdftitle={Subword counting and the incidence algebra},
  colorlinks=true,
  citecolor=blue,
  urlcolor=blue,
  linkcolor=violet,
  pdfborder={0 0 0}}
\usepackage{tikz}
\usetikzlibrary{matrix,arrows,patterns,calc}

\usepackage[euler-digits,euler-hat-accent]{eulervm}
\usepackage{concrete}

\usepackage{stmaryrd}

\newtheorem{theorem}{Theorem}
\newtheorem{lemma}[theorem]{Lemma}
\newtheorem{proposition}[theorem]{Proposition}
\newtheorem{corollary}[theorem]{Corollary}

\theoremstyle{definition}

\newtheorem{example}[theorem]{Example}

\newcommand{\NN}{\mathbb{N}}
\newcommand{\ZZ}{\mathbb{Z}}
\newcommand{\QQ}{\mathbb{Q}}

\newcommand{\dual}[1]{{#1}^{\star}}
\newcommand{\sym}{\mathcal{S}}
\newcommand{\eps}{\epsilon}

\newcommand{\coeff}[2]{\langle #1,#2 \rangle}

\renewcommand{\o}{\mathfrak{o}}

\newcommand{\val}[1]{\llbracket #1 \rrbracket}
\newcommand*{\Cdot}{\raisebox{-0.4ex}{\scalebox{1.2}{$\cdot$}}}

\DeclareMathOperator{\Int}{\mathrm{Int}}
\DeclareMathOperator{\xor}{\text{\sc xor}}
\DeclareMathOperator{\band}{\text{\sc and}}
\DeclareMathOperator{\bor}{\text{\sc or}}

\title{Subword counting and the incidence algebra}
\author{Anders Claesson}
\address{Department of Computer and Information Sciences,
    University of Strathclyde, Glasgow, UK}
\date{\today}

\begin{document}

\begin{abstract}
  The Pascal matrix, $P$, is an upper diagonal matrix whose entries are
  the binomial coefficients. In 1993 Call and Velleman demonstrated that
  it satisfies the beautiful relation $P=\exp(H)$ in which $H$ has the
  numbers 1, 2, 3, etc.\ on its superdiagonal and zeros elsewhere. We
  generalize this identity to the incidence algebras $I(A^*)$ and
  $I(\sym)$ of functions on words and permutations, respectively.  In
  $I(A^*)$ the entries of $P$ and $H$ count subwords; in $I(\sym)$ they
  count permutation patterns. Inspired by vincular permutation patterns
  we define what it means for a subword to be restricted by an auxiliary
  index set $R$; this definition subsumes both factors and (scattered)
  subwords.
  We derive a theorem for words corresponding to the Reciprocity Theorem
  for patterns in permutations: Up to sign, the coefficients in
  the Mahler expansion of a function counting subwords restricted by the
  set $R$ is given by a function counting subwords restricted by the
  complementary set $R^c$.
\end{abstract}

\maketitle
\thispagestyle{empty}

\section{Introduction}

Let us start by recalling some terminology from the theory of
posets; our presentation follows that of Stanley~\cite[sec.\ 3]{EC1}. Let
$(Q,\leq)$ be a poset. For $x,y\in Q$, let $[x,y]=\{ z : x\leq z\leq
y\}$ denote the \emph{interval} between $x$ and $y$. All posets
considered in this paper will be \emph{locally finite}: that is, each
interval $[x,y]$ is finite. We say that $y$ \emph{covers} $x$ if $[x,y]
= \{x,y\}$.

Let $\Int(Q)=\{(x,y) \in Q\times Q : x \leq y\}$; it is a set that can
be thought of as representing the intervals of $Q$. The \emph{incidence
  algebra}, $I(Q)$, of $(Q,\leq)$ over some field $K$ is the $K$-algebra
of all functions $F:\Int(Q) \to K$ with the usual structure as a vector
space over $K$ and multiplication (convolution) defined by
$$(FG)(x,y) = \sum_{x\leq z\leq y}F(x,z)G(z,y),
$$ and identity, $\delta$, defined by $\delta(x,y) = 1$ if $x = y$,
and $\delta(x,y) = 0$ if $x \neq y$. Let $\val{\psi}$ be the Iverson
bracket; it denotes a number that is $1$ if the statement $\psi$ is
satisfied, and $0$ otherwise. With this notation we have
$\delta(x,y)=\val{x=y}$. Two other prominent elements of the incidence
algebra are $\zeta(x,y)=1$ and $\eta(x,y)=\val{y\text{ covers }x}$.  One
can show (see Stanley~\cite[sec. 3]{EC1}) that $f\in I(Q)$ has a
two-sided inverse $f^{-1}$ if and only if $f(x,x)\neq 0$ for all $x\in
Q$. For instance, $\zeta$ is invertible and its inverse,
$\mu=\zeta^{-1}$, is known as the \emph{M\"obius function}.

The $(i,j)$ entry of the upper triangular $n\times n$ \emph{Pascal
  matrix} $P_n$ is defined as $\binom{j}{i}$. The superdiagonal matrix
$H_n$ is defined by letting the entries on the superdiagonal $j=i+1$
be $j$. Further, let $I_n$ denote the $n\times n$ identity matrix.
In 1993 Call and Velleman~\cite{CaVe1993} calculated the powers of $H_n$
and showed that $e^{H_n} = P_n$, where $e^{H_n}= I_n + H_n +
H_n^2/2! + H_n^3/3! + \cdots$.  As an example, for $n=3$
we have
$$
\exp\left(
\begin{bmatrix}
  0 & 1 & 0 & 0 \\
    & 0 & 2 & 0 \\
    &   & 0 & 3 \\
    &   &   & 0
\end{bmatrix}
\right) =
\begin{bmatrix}
  1 & 1 & 1 & 1 \\
    & 1 & 2 & 3 \\
    &   & 1 & 3 \\
    &   &   & 1
\end{bmatrix}.
$$

If we consider $\NN$ as a poset under the usual order on integers we may
view the result of Call and Velleman as an identity in the incidence
algebra $I(\NN)$: We have
\begin{equation}\label{eq:call-velleman}
  e^H = P,
\end{equation}
where $H(i,j)=j\eta(i,j)=j\val{j=i+1}$
and $P(i,j)=\binom{j}{i}$ are elements of $I(\NN)$, and
$$e^{H}= \delta + H + \frac{1}{2!}H^2 + \frac{1}{3!}H^3+\cdots.
$$
Note that the restriction of $H$ to integers smaller than some fixed
bound is nilpotent, and hence each entry in $e^H$ is a finite sum (in
fact, a single term). We shall generalize \eqref{eq:call-velleman} to
the incidence algebra over a poset of words, but first we need a few
more definitions.

Let $[n] = \{1,2,\dots,n\}$.  Let $A^*$ be the set of words with letters
in a given finite set (alphabet) $A$. A word $u=a_1 a_2 \dots a_k$, with
$a_i\in A$, is said to be a \emph{subword} of a word $v=b_1 b_2\dots
b_n$, with $b_i\in A$, if there is an order-preserving injection
$\varphi: [k]\to [n]$ such that $a_i = b_{\varphi(i)}$ for all $i\in
[k]$; that is, if $u$ is a subsequence of $v$. By an \emph{occurrence}
of $u$ (as a subword) in $v$ we shall mean the function $\varphi$, often
represented by the set $\{\varphi(i) : i \in [k]\}$.  For instance, the
word $baacbab$ contains four occurrences of $aab$. Following
Eilenberg~\cite[VIII.10]{Ei1976} we let $\binom{v}{u}$ denote the number
of occurrences of $u$ as a subword in $v$.  Note that if $a$ is a
letter, then $\binom{a^n}{a^k} = \binom{n}{k}$, where the right-hand
side is the ordinary binomial coefficient.

We will consider $A^*$ a poset under the \emph{subword order}, also
called division order. That is, $u\leq v$ in $A^*$ if and only if
$\binom{v}{u} > 0$. For $A=\{a,b\}$ the first four levels of the Hasse
diagram of this infinite poset are depicted in Figure~\ref{fig:words}.

\begin{figure}
$$
\begin{tikzpicture}[>=stealth, node distance=1.54cm]
  \node (aaa) {$aaa$};
  \node (aab) [right of=aaa] {$aab$};
  \node (aba) [right of=aab] {$aba$};
  \node (abb) [right of=aba] {$abb$};
  \node (baa) [right of=abb] {$baa$};
  \node (bab) [right of=baa] {$bab$};
  \node (bba) [right of=bab] {$bba$};
  \node (bbb) [right of=bba] {$bbb$};
  \node (aa) [below right of=aab] {$aa$};
  \node (ab) [below right of=aba] {$ab$};
  \node (ba) [below left  of=bab] {$ba$};
  \node (bb) [below left  of=bba] {$bb$};
  \node (a) [below right of=aa] {$a$};
  \node (b) [below left  of=bb] {$b$};
  \node (e) at ($(a)!0.5!(b)-(0,0.8)$) {$\eps$};
  \draw (e)--(a) (e)--(b);
  \draw (a)--(aa) (a)--(ab) (a)--(ba);
  \draw (b)--(ab) (b)--(ba) (b)--(bb);
  \draw (aa)--(aaa) (aa)--(aab) (aa)--(aba) (aa)--(baa);
  \draw (ab)--(aab) (ab)--(aba) (ab)--(abb) (ab)--(bab);
  \draw (ba)--(aba) (ba)--(baa) (ba)--(bab) (ba)--(bba);
  \draw (bb)--(abb) (bb)--(bab) (bb)--(bba) (bb)--(bbb);
\end{tikzpicture}
$$
\caption{The first four levels of the poset of words over $\{a,b\}$ with
  respect to the subword/division order.}\label{fig:words}
\end{figure}
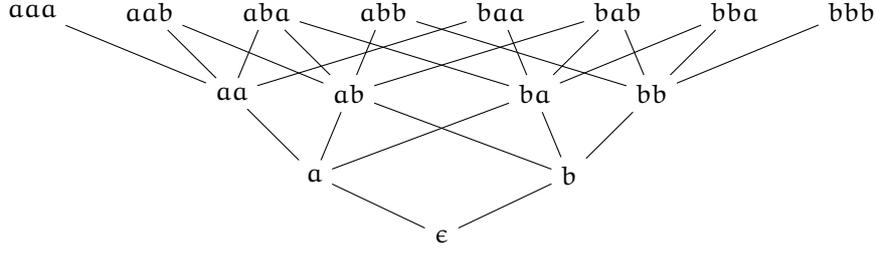

We are now ready to state the promised generalization of
\eqref{eq:call-velleman}. Recall that $\eta(x,y) = \val{y\text{ covers
  }x}$. For the poset $\NN$ we hence have $\eta(i,j)=1$ if, and only if,
$j=i+1$. For the poset $A^*$ we have $\eta(u,v)=1$ if, and only if, $u$
is a subword of $v$ and $|v|=|u|+1$.  Theorem~\ref{thm:P=e^H}, which we
prove in Section~\ref{sec:proof-of-thm}, then states the following:

\begin{quote}
  Define the elements $P$ and $H$ of the incidence algebra $I(A^*)$
  by $P(u,v) = \binom{v}{u}$ and $H(u,v) = \binom{v}{u}\eta(u, v)$.
  Then $P = e^H$.
\end{quote}

Note that if $A$ is a singleton, and $u$ and $v$ are words in $A^*$ of
length $n$ and $k$, respectively, then $u\leq v$ in $A^*$ precisely then
$k\leq n$ in $\NN$. Further, $P(u,v)=\binom{n}{k}$ and
$H(u,v)=\binom{n}{k}\val{n=k+1} = n\val{n=k+1}$.  Thus,
Equation~\ref{eq:call-velleman} is a special case of
Theorem~\ref{thm:P=e^H}.

We can also use the relation $P=e^H$ to derive a simple formula for the
powers of $P$. The following statement is Theorem~\ref{thm:powers-of-P} of
Section~\ref{sec:powers-of-P}.

\begin{quote}
  For $u,v\in A^*$ and any integer $d$ we have $P^d(u,v) =
  d^{|v|-|u|}P(u,v)$.
\end{quote}

Let us for a moment turn to permutations instead of words. Let $\sym_n$
be the set of permutations of $[n]$, and let $\sym=\cup_{n\geq 0}
\sym_n$. Two sequences of integers $a_1\dots a_k$ and $b_1\dots b_k$ are
\emph{order-isomorphic} if for every $i,j\in [k]$ we have $a_i<a_j
\Leftrightarrow b_i<b_j$. A permutation $\sigma\in\sym_k$ is said to be
contained in a permutation $\pi\in\sym_n$ if there is an
order-preserving injection $\varphi: [k]\to [n]$ such that
$\pi(\varphi(1))\pi(\varphi(2))\dots\pi(\varphi(k))$ is order-isomorphic
to $\sigma$. In this context, $\sigma$ is often called a
\emph{pattern}. By an \emph{occurrence} of $\sigma$ in $\pi$ we shall
mean the function $\varphi$, or the set $\{\varphi(i) : i \in
[k]\}$. For instance, the permutation $43152$ contains two
occurrences of the pattern $231$, namely $\{1,4,5\}$ and $\{2,4,5\}$. Let
$\binom{\pi}{\sigma}$ be the number of occurrences of $\sigma$ in
$\pi$. Note that $\binom{12\dots n}{12\dots k} = \binom{n}{k}$, where
the right-hand side is the ordinary binomial coefficient.


Much like with $A^*$ we endow $\sym$ with a poset structure by
postulating that $\sigma\leq\pi$ if, and only if,
$\binom{\pi}{\sigma}>0$. The first four levels of the Hasse diagram of
this infinite poset are depicted in Figure~\ref{fig:perms}. It turns out
that the $P=e^H$ equation in $I(A^*)$ has a natural counterpart in
$I(\sym)$:

\begin{figure}
$$
\begin{tikzpicture}[>=stealth,node distance=1.6cm]
  \node (123) {123};
  \node (132) [right of=123] {132};
  \node (213) [right of=132] {213};
  \node (231) [right of=213] {231};
  \node (312) [right of=231] {312};
  \node (321) [right of=312] {321};
  \node (12) [below right of=132] {12};
  \node (21) [below left  of=312] {21};
  \node (1) at ($(12)!0.5!(21)-(0,0.7)$) {1};
  \node (e) at ($(12)!0.5!(21)-(0,1.6)$) {$\eps$};
  \draw (e)--(1);
  \draw (1)--(12) (1)--(21);
  \draw (12)--(123) (12)--(132) (12)--(213) (12)--(231) (12)--(312);
  \draw (21)--(132) (21)--(213) (21)--(231) (21)--(312) (21)--(321);
\end{tikzpicture}
$$
\caption{The first four levels of the poset of permutations with respect
  to pattern containment.}\label{fig:perms}
\end{figure}
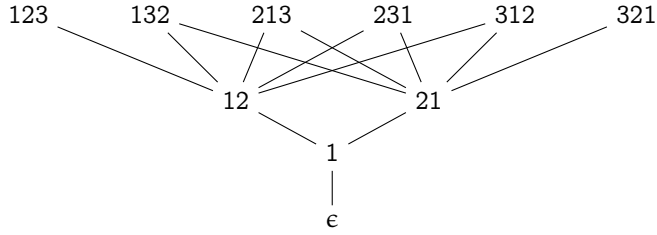

\begin{quote}
  Define the elements $P$ and $H$ of the incidence algebra $I(\sym)$
  by $P(\sigma,\pi) = \binom{\pi}{\sigma}$ and $H(\sigma,\pi) =
  \binom{\pi}{\sigma}\eta(\sigma,\pi)$. Then $P = e^H$.
\end{quote}



In 2011 Bränd\'en and the present author (see ~\cite{BrCl2011})
presented a ``Reciprocity Theorem'' for so called \emph{mesh patterns},
a class of permutation patterns introduced in the same paper.
Intuitively, an occurrence of a mesh pattern is an occurrence in the
sense defined above with additional restrictions on the relative
position of the entries of the occurrence. \emph{Vincular patterns}
(Babson and Steingrímsson, 2000~\cite{BaSt2000}) as well as
\emph{bivincular patterns} (Bousquet-M{\'e}lou et al.,
2010~\cite{BCDK2010}) can be seen as special mesh patterns. In this
paper we shall not need the general definition of a mesh pattern, rather
we just give the more specialized definition of a vincular pattern: Let
$\pi\in \sym_n$. Let $\sigma\in \sym_k$ and $R\subseteq [0,k]$. The
pair $p=(\sigma, R)$ is called a \emph{vincular pattern} and an
occurrence of $p$ in $\pi$ is an order-preserving injection $\varphi:
[k]\to [n]$ such that
$\pi(\varphi(1))\pi(\varphi(2))\dots\pi(\varphi(k))$ is order-isomorphic
to $\sigma$ and, for all $i\in R$, $\varphi(i+1) = \varphi(i)+1$, where
$\varphi(0)=0$ and $\varphi(k+1)=n+1$ by convention. Let $\binom{\pi}{p}$
denote the number of occurrences of $p$ in $\pi$. We will find it convenient
to write $\binom{\pi}{\sigma,R}$ rather than the typographically awkward
$\binom{\pi}{(\sigma,R)}$. For instance,
$$
\binom{43152}{231,\{0\}} = 1,\;
\binom{43152}{231,\{1\}} = 0,\;
\binom{43152}{231,\{2\}} = 2 \;\,\text{ and }\;\,
\binom{43152}{231,\{0,3\}} = 1.
$$

Any function $f : \sym\to\QQ$ can be written as a unique, but typically
infinite, linear combination of functions $\binom{\Cdot}{\sigma}$. To see
this, let $I(\sym)$ act on the function space
$\QQ^{\sym}$ by
$$(f\ast F)(\pi) = \sum_{\sigma \leq \pi}f(\sigma)F(\sigma,\pi).
$$
Thus, $f = \sum_\sigma c(\sigma)\binom{\Cdot}{\sigma}$ is equivalent to
$f = c \ast P$ and, since $P$ is invertible, $c = f \ast P^{-1}$.
Viewing $\binom{\Cdot}{p}$ as a function in $\QQ^{\sym}$, the
Reciprocity Theorem ~\cite[Thm. 1]{BrCl2011} then gives the coefficient
$c(\sigma)$ of $\binom{\Cdot}{\sigma}$ in terms of occurrences of the
\emph{dual} pattern $\dual{p} = (\sigma, [0,|\sigma|]\setminus R)$.
For any permutation $\pi$, we have
$$
\binom{\pi}{p} = \sum_{\upsilon\in A^*}\coeff{p}{\upsilon}\binom{\pi}{\upsilon},\;
\text{ where }
\coeff{p}{\upsilon} = (-1)^{|\upsilon|-|\sigma|}\binom{\upsilon}{\dual{p}}.
$$

In Section~\ref{sec:reciprocity} we present a corresponding theorem for
words. To make this explicit we need a definition: Let $u=a_1 a_2 \dots
a_k$, with $a_i\in A$, and let $v=b_1 b_2\dots b_n$, with $b_i\in A$.
Let $p=(u,R)$ with $R\subseteq [0,k]$. An \emph{occurrence} of $p$ in
$v$ is an order-preserving injection $\varphi: [k]\to [n]$ such that
$a_i = b_{\varphi(i)}$ for all $i\in [k]$, and, for all $i\in R$,
$\varphi(i+1) = \varphi(i)+1$, with the convention that $\varphi(0)=0$
and $\varphi(k+1)=n+1$. Let $\binom{v}{p}$ denote the number of
occurrences of $p$ in $v$. For example,
$$
\binom{baacbab}{aab, \{0\}} = 0,\;
\binom{baacbab}{aab, \{1\}} = 2,\;
\binom{baacbab}{aab, \{1,2\}} = 0\;\,\text{ and }\;\,
\binom{baacbab}{aab, \{1,3\}} = 1.
$$

We would like to draw attention to the following three important special
cases of the above definition:
\begin{enumerate}
\item if $R=\emptyset$ then $\binom{v}{p} = \binom{v}{u}$;
\item if $R=[1,k-1]$ then $\binom{v}{p}$ is the number of occurrences
  of $u$ as a factor in $v$;
\item if $R=[0,k]$ then $\binom{v}{p} = \delta(u,v)$.
\end{enumerate}
With regard to the second item, a word $u$ is a \emph{factor} of another
word $v$ if there are (possibly empty) words $x$ and $y$ such that $v =
xuy$.

We are now ready to state the Reciprocity Theorem for $I(A^*)$, which is
Theorem~\ref{thm:reciprocity} of Section~\ref{sec:reciprocity}.  For $u$
a word and $R\subseteq [0,|u|]$ let $p = (u,R)$, and let $\dual{p} = (u,
[0,|u|]\setminus R)$. Then, for any word $w$, we have
$$
\binom{w}{p} = \sum_{v\in A^*}\coeff{p}{v}\binom{w}{v},\;
\text{ where }
\coeff{p}{v} = (-1)^{|v|-|u|}\binom{v}{\dual{p}}.
$$

\section{$P$ is the exponential of $H$}\label{sec:proof-of-thm}

Let $P$ and $H$, in $I(A^*)$, be defined as in the introduction. The key to
proving that $P = e^H$ will be the simple formula for the powers of $H$
provided by the Lemma~\ref{lemma:H^l} below. Its proof will use some
terminology we now define: A subset $C$ of a poset $Q$ is called a chain
if for any pair of elements $x$ and $y$ in the subposet $C$ of $Q$ we
have $x\leq y$ or $y\leq x$. A chain $x_0<x_1<\dots < x_n$ is
\emph{saturated} if $x_i$ covers $x_{i-1}$ for each $i\in [n]$.

\begin{lemma}\label{lemma:H^l}
  For $u,v\in A^*$ we have
  \begin{equation}\label{eq:H^k}
    H^\ell(u,v) = \ell! \binom{v}{u} \val{\,|v|=|u|+\ell\,}.
  \end{equation}
\end{lemma}
\begin{proof}
  We shall give a combinatorial proof. Let $n=|v|$. Expanding the
  left-hand side, $H^{\ell}(u,v)$, we get
  $$\sum H(x_0,x_1)H(x_1,x_2)\dots H(x_{\ell-1},x_{\ell}),
  $$
  where the sum is over all saturated chains
  $u=x_0<x_1<\dots<x_{\ell-1}<x_\ell = v$. Assume that we are given such
  a chain. By transitivity we may consider each $x_i$ as a subword of
  $v$; let $\o_i\subseteq [n]$ be a set of indices that determine the
  subword $x_i$ in $v$. Because the chain $x_0<x_1<\dots<x_\ell$ is
  saturated, so is the chain
  $\o_0\subset\o_1\subset\dots\subset\o_\ell$, and thus
  $\o_{i+1}\setminus \o_i$ is a singleton. Suppose that
  $\o_{i+1}\setminus \o_i = \{m_i\}$. Then $m_0m_1\dots m_{\ell-1}$ is a
  permutation of $\o_\ell\setminus \o_0$. To summarize, any given
  saturated chain $u=x_0<x_1<\dots<x_{\ell-1}<x_\ell = v$ together with
  occurrences $\o_0$, $\o_1$, \dots, $\o_\ell$ of $x_0$, $x_1$, \dots,
  $x_\ell$, respectively, in $v$, determine a pair consisting of a
  permutation of the $\ell$ elements of $\o_\ell\setminus \o_0$ and an
  occurrence $\o_0$ of $u$ in $v$.

  Conversely, let $\o_\ell = [n]$ be the index set of $v$, and let
  $\o_0$ be an occurrence of $u$ in $v$. Let $m_0 m_1 \dots m_{\ell-1}$
  be a permutation of $\o_\ell\setminus \o_0$. We construct a saturated
  chain $u=x_0<x_1<\dots<x_{\ell-1}<x_\ell = v$ and occurrences $\o_0$,
  $\o_1$, \dots, $\o_\ell$ of $x_0$, $x_1$, \dots, $x_\ell$,
  respectively, in $v$ by letting $\o_{i+1} = \o_i \cup \{m_i\}$.
\end{proof}

\begin{theorem}\label{thm:P=e^H}
  Define the elements $P$ and $H$ of the incidence algebra $I(A^*)$
  by $P(u,v) = \binom{v}{u}$ and $H(u,v) = \binom{v}{u}\eta(u, v)$.
  Then $P = e^H$.
\end{theorem}

\begin{proof}
  By Lemma~\ref{lemma:H^l} we have
  \begin{equation*}
  (e^H)(u,v)
  = \sum_{\ell\geq 0}\frac{1}{\ell!}H^\ell(u,v)
  = \sum_{\ell\geq 0}\binom{v}{u} \val{\, |v|=|u|+\ell\,}
  = \binom{v}{u}.\qedhere
  \end{equation*}
\end{proof}

\begin{example}
  Let us illustrate Lemma~\ref{lemma:H^l}
  and its proof for $A=\{a,b\}$, $\ell=2$, $u=ab$ and $v=aaba$. The
  left-hand side of \eqref{eq:H^k} is
  \begin{equation*}
    H(ab,aab)H(aab,aaba) \,+\, H(ab,aba)H(aba,aaba) = 2\cdot 1 \,+\, 1\cdot 2,
  \end{equation*}
  while the right-hand side is $2!\binom{aaba}{ab} = 2\cdot 2$. The
  following table gives the saturated chains and corresponding pairs as
  in the proof of Lemma~\ref{lemma:H^l}:
  $$
  \begin{array}{ccc}
    \{1,3\}\subset \{1,2,3\} \subset \{1,2,3,4\} && (24,\{1,3\}); \\
    \{2,3\}\subset \{1,2,3\} \subset \{1,2,3,4\} && (14,\{2,3\}); \\
    \{1,3\}\subset \{1,3,4\} \subset \{1,2,3,4\} && (42,\{1,3\}); \\
    \{2,3\}\subset \{2,3,4\} \subset \{1,2,3,4\} && (41,\{2,3\}).
  \end{array}
  $$
\end{example}

\begin{example}
  If we only consider words shorter than some fixed length, then the
  poset of words is finite and the elements of the incidence algebra can
  be seen as upper triangular matrices. In the following example we
  consider words over $\{a,b\}$ of length at most 2. Then
  $$H =
  \begin{bmatrix}
    \;0 & \binom{a}{\eps} & \binom{b}{\eps} & 0 & 0 & 0 & 0       \smallskip\\
        & 0 & 0 & \binom{aa}{a}& \binom{ab}{a}& \binom{ba}{a}& 0  \medskip\\
        &   & 0 & 0 & \binom{ab}{b}& \binom{ba}{b}& \binom{bb}{b} \smallskip\\
        &   &   & 0 & 0 & 0 & 0 \\
        &   &   &   & 0 & 0 & 0 \\
        &   &   &   &   & 0 & 0 \\
        &   &   &   &   &   & 0
  \end{bmatrix} =
  \begin{bmatrix}
    0 & 1 & 1 & 0 & 0 & 0 & 0 \\
      & 0 & 0 & 2 & 1 & 1 & 0 \\
      &   & 0 & 0 & 1 & 1 & 2 \\
      &   &   & 0 & 0 & 0 & 0 \\
      &   &   &   & 0 & 0 & 0 \\
      &   &   &   &   & 0 & 0 \\
      &   &   &   &   &   & 0
  \end{bmatrix}.
  $$
  Note that all eigenvalues of $H$ are zero and thus $H$ is
  nilpotent.
  To be more precise, $H^i$ is the zero matrix for all $i\geq 3$, and
  $$
    e^H
    = I + H + \frac{1}{2}H^2
    =
    \begin{bmatrix}
      1 & 1 & 1 & 1 & 1 & 1 & 1 \\
        & 1 & 0 & 2 & 1 & 1 & 0 \\
        &   & 1 & 0 & 1 & 1 & 2 \\
        &   &   & 1 & 0 & 0 & 0 \\
        &   &   &   & 1 & 0 & 0 \\
        &   &   &   &   & 1 & 0 \\
        &   &   &   &   &   & 1
    \end{bmatrix},
  $$
  which, in agreement with Theorem~\ref{thm:P=e^H}, is equal to
  $$
    \begin{bmatrix}
      \binom{\eps}{\eps} & \binom{a}{\eps} & \binom{b}{\eps} & \binom{aa}{\eps}
      & \binom{ab}{\eps}& \binom{ba}{\eps}& \binom{bb}{\eps}
      \medskip\\
        &\binom{a}{a} & 0 &\binom{aa}{a} &\binom{ab}{a} &\binom{ba}{a} & 0
      \medskip\\
        &   &\binom{b}{b} & 0 & \binom{ab}{b} &\binom{ba}{b} &\binom{bb}{b}
      \medskip\\
        &   &   & \binom{aa}{aa}& 0 & 0 & 0 \medskip\\
        &   &   &   & \binom{ab}{ab}& 0 & 0 \medskip\\
        &   &   &   &   & \binom{ba}{ba}& 0 \medskip\\
        &   &   &   &   &   & \binom{bb}{bb}
    \end{bmatrix}.
  $$
\end{example}

The powerset of $[n]$ with respect to the subset relation is a poset
called the \emph{boolean algebra} and is denoted $B_n$. We note that one
part of the proof of Lemma~\ref{lemma:H^l} above is the well known fact
that the number of saturated chains from $\emptyset$ to $[n]$ in $B_n$
is $n!$. Recall that $\zeta(S,T)=1$ and
$\eta(S,T)=\val{T\text{ covers }S}$.
The proof of the following proposition, which we omit, is
very similar to but slightly easier than the proof of
Lemma~\ref{lemma:H^l}.

\begin{proposition}\label{prop:e^zeta=eta}
  In the incidence algebra $I(B_n)$ we have $e^\eta=\zeta$.
\end{proposition}

The simple formula $\mu(S,T) = (-1)^{|T\setminus S|}$ for the M\"obius
function of the boolean algebra can be derived from the isomorphism
$B_n\cong \mathbf{2}^n$ using the so called product
rule~\cite[Ex. 3.8.3]{EC1}. This formula also follows from
Proposition~\ref{prop:e^zeta=eta}: Let $S,T\in B_n$. As in the proof of
Lemma~\ref{lemma:H^l}, the number of saturated chains starting at $S$
and ending in $T$ is $\ell!$, where $\ell=|T\setminus S|$. The number of
such chains is also, however, $\eta^\ell(S,T)$. To be more precise
$\eta^\ell(S,T) = \ell!\val{\,|T|=|S|+\ell\,}$, and thus
$$\mu(S,T) = \zeta^{-1}(S,T) = (e^{-\eta})(S,T) =
\sum_{\ell\geq 0}\frac{(-1)^{\ell}}{\ell!}\eta^\ell(S,T) = (-1)^\ell.
$$

Let us now consider the permutation poset $\sym$. It is easy to see that
the proof of Lemma~\ref{lemma:H^l} can be modified to apply to the
setting of the incidence algebra $I(\sym)$, and we close this section by
stating the counterpart of Theorem~\ref{thm:P=e^H} that follows from
doing so.

\begin{theorem}
  Define the elements $P$ and $H$ of the incidence algebra $I(\sym)$
  by $P(\sigma,\pi) = \binom{\pi}{\sigma}$ and $H(\sigma,\pi) =
  \binom{\pi}{\sigma}\eta(\sigma,\pi)$. Then $P = e^H$.
\end{theorem}

\begin{example}
  If we only consider permutations shorter than some fixed length, then
  the poset of words is finite and the elements of the incidence algebra
  can be seen as upper triangular matrices. In the following example we
  consider permutations of length at most 2. We have
  \begin{align*}
  H =
  \begin{bmatrix}
    \;0 & \binom{1}{\eps} & 0             & 0 \smallskip \\
        & 0               & \binom{12}{1} & \binom{21}{1} \smallskip \\
        &                 & 0             & 0 \\
        &                 &               & 0
  \end{bmatrix} &=
  \begin{bmatrix}
    0 & 1 & 0 & 0 \\
      & 0 & 2 & 2 \\
      &   & 0 & 0 \\
      &   &   & 0
  \end{bmatrix}\quad\text{and}\\
  P = I + H + \frac{1}{2}
  \begin{bmatrix}
    0 & 0 & 2 & 2 \\
      & 0 & 0 & 0 \\
      &   & 0 & 0 \\
      &   &   & 0
  \end{bmatrix} &=
  \begin{bmatrix}
    1 & 1 & 1 & 1 \\
      & 1 & 2 & 2 \\
      &   & 1 & 0 \\
      &   &   & 1
  \end{bmatrix} =
  \begin{bmatrix}
    \binom{\eps}{\eps} &\binom{1}{\eps} &\binom{12}{\eps} &\binom{21}{\eps} \medskip\\
                       &\binom{1}{1}    &\binom{12}{1}    &\binom{21}{1}    \medskip\\
                       &                &\binom{12}{12}   & 0               \smallskip\\
                       &                &                 &\binom{21}{21}
  \end{bmatrix}.
  \end{align*}
\end{example}

\section{Powers of $P$}\label{sec:powers-of-P}

From Sakarovitch and Simon~\cite[Corollary 6.3.8]{Lo1983} we learn that
\begin{equation}\label{eq:inverse}
  \sum_{w\in A^*} (-1)^{|u|+|v|}\binom{w}{u}\binom{v}{w} = \delta(u,v).
\end{equation}
Their proof uses the so called Magnus transformation, the algebra
endomorphism $a\mapsto a+1$ of $\ZZ\langle A\rangle$. An analog of this
result for permutations is a consequence of the Reciprocity Theorem for
mesh patterns~\cite[Corollary~2]{BrCl2011}; more recently,
Vargas~\cite{Va2014} has given a proof of the same result using an
analog of the Magnus transformation for permutations.

In the context of the incidence algebra, Equation \ref{eq:inverse}
states that $P^{-1}(u,v) = (-1)^{|v|-|u|}P(u,v)$. We could use
Theorem~\ref{thm:P=e^H} to give a simple alternative proof of
this. Indeed, $P=e^H$ implies that $P^{-1}=e^{-H}$. We can in fact prove
something stronger:

\begin{theorem}\label{thm:powers-of-P}
  For any words $u$ and $v$ in $A^*$, and any integer $d$, we have
  $$P^d(u,v) = d^{|v|-|u|}P(u,v).
  $$
  If $d\neq 0$ and $D_d(u,v) = d^{|u|}\delta(u,v)$, then an equivalent
  way of stating this result is
  $P^d = D_d^{-1} P D_d = D_{1/d} P D_d$. In particular, $P^{-1} = D_{-1}PD_{-1}$.
\end{theorem}

\begin{proof}
  We have
  \begin{align*}
    P^d(u,v)
    &= (e^H)^d(u,v) &&\text{by Theorem~\ref{thm:P=e^H}} \\
    &= e^{dH}(u,v)\\
    &= \sum_{\ell\geq 0}\frac{1}{\ell!}(dH)^\ell(u,v) \\
    &= \sum_{\ell\geq 0}\frac{d^\ell}{\ell!}H^\ell(u,v) \\
    &= \sum_{\ell\geq 0}d^{\ell}P(u,v)\val{\,|v|=|u|+\ell\,}
    && \text{by Lemma~\ref{lemma:H^l}}\\
    &= d^{|v|-|u|}P(u,v).&& \qedhere
  \end{align*}
\end{proof}

\begin{example}
  If we restrict $P$ to words of length at most $n$ and let $A$ be a
  singleton alphabet, then the function $P$ reduces to the $n\times n$
  Pascal matrix $P_n$. Thus, a corollary to
  Theorem~\ref{thm:powers-of-P} is that the Pascal matrix satisfies
  $P^d_n(i,j) = d^{j-i}P_n(i,j)$. This is, however, a known result due
  to Call and Velleman~\cite{CaVe1993}.
\end{example}

A \emph{multichain} in a poset $Q$ is a multiset whose underlying set is
a chain in $Q$.  It is well known, and easy to prove, that the number of
multichains
$$
\emptyset=S_0\subseteq S_1 \subseteq S_2 \subseteq\dots\subseteq
S_{d-1}\subseteq S_d = [\ell]
$$
in the boolean algebra $B_\ell$ is $d^\ell$. Indeed, such a multichain
is uniquely specified by a function $t : [\ell] \to [d]$ where $t(i)$ is
the smallest $j\in [d]$ for which $i\in S_j$. We use this fact in the
following alternative proof of Theorem~\ref{thm:powers-of-P}.

\begin{proof}[Combinatorial proof of Theorem~\ref{thm:powers-of-P}]
  Let $n=|v|$ and $\ell = |v|-|u|$. Expanding the
  left-hand side, $P^{d}(u,v)$, we get
  $$\sum P(x_0,x_1)P(x_1,x_2)\dots P(x_{d-1},x_{d}),
  $$
  where the sum is over all multichains $u=x_0\leq x_1\leq \dots \leq
  x_{d-1}\leq x_d = v$.  Assume that we are given such a chain. By
  transitivity we may consider each $x_i$ as a subword of $v$; let
  $\o_i\subseteq [n]$ be a set of indices that determine the subword
  $x_i$ in $v$. Let $S_0 = \o_0\setminus \o_0 = \emptyset,
  S_1=\o_1\setminus \o_0$, $S_2=\o_2\setminus \o_0$, etc, and let
  $S=\o_d\setminus\o_0$. Then $\emptyset=S_0\subseteq S_1 \subseteq S_2
  \subseteq\dots\subseteq S_{d-1}\subseteq S_d = S$ is a multichain in
  the boolean algebra on $S$. As in the paragraph preceding this proof,
  let $t : S\to [d]$ be the function specifying that multichain. Because
  $|S|=\ell$, the number of such functions, and therefore also the
  number of multichains, is $d^\ell$.

  To summarize, any given multichain $u=x_0\leq x_1\leq \dots \leq
  x_{d-1}\leq x_d = v$ together with occurrences $\o_0$, $\o_1$, \dots,
  $\o_d$ of $x_0$, $x_1$, \dots, $x_d$, respectively, in $v$, determine
  a pair consisting of a function $t:\o_d\setminus\o_0 \to [d]$ and an
  occurrence $\o_0$ of $u$ in $v$. The total number of such pairs is, of
  course, $d^{|v|-|u|}P(u,v)$.  It is also easy to see how this could be
  reversed and thus the procedure described is a bijection.
\end{proof}

Either of the two proofs we have presented for
Theorem~\ref{thm:powers-of-P} can easily be adopted to the permutation
setting. Thus we have rediscovered the following result, which was known
to Petter Bränd\'en already in 2002 (personal communication).

\begin{theorem}[Bränd\'en 2002]
  For $\sigma,\pi\in\sym$ and any integer $d$ we have $P^d(\sigma,\pi)
  = d^{|\pi|-|\sigma|}P(\sigma,\pi)$.
\end{theorem}

\section{A reciprocity theorem}\label{sec:reciprocity}

Let $u=a_1 a_2 \dots a_k$, with $a_i\in A$, and let $v=b_1 b_2\dots
b_n$, with $b_i\in A$.  Let $p=(u,R)$ with $R\subseteq [0,k]$. Recall
that an \emph{occurrence} of $p$ in $v$ is an order-preserving injection
$\varphi: [k]\to [n]$ such that $a_i = b_{\varphi(i)}$ for all $i\in
[k]$, and, for all $i\in R$, $\varphi(i+1) = \varphi(i)+1$, with the
convention that $\varphi(0)=0$ and $\varphi(k+1)=n+1$.  The number of
occurrences of $p$ in $v$ is denoted by $\binom{v}{p}$ or
$\binom{v}{u,R}$.

\begin{lemma}[A generalization of Pascal's formula]\label{lemma:pascal}
  Let $u$ and $v$ be words, and let $a$ and $b$ be letters. Let $k=|ub|$
  and let $R\subseteq [0,k]$. Then we have
  $$\binom{va}{ub,R} =
  \val{k\notin R}\binom{v}{ub,R} 
  + \val{a=b}\binom{v}{u,R\setminus \{k\}}.
  $$
\end{lemma}
\begin{proof}
  An occurrence of $(ub,R)$ in $va$ may match the last letter of $ub$
  with the last letter of $va$; the number of such occurrences is
  $\val{a=b}\binom{v}{u,R\setminus \{k\}}$. If $k\in R$ this is the only option: we
  have to match the last letter of $ub$ with the last letter of
  $va$. If, on the other hand, $k\notin R$, then we have additional
  occurrences, namely those that do not involve the last letter of
  $va$, and there are $\binom{v}{ub,R}$ such occurrences.
\end{proof}

\begin{theorem}[Reciprocity]\label{thm:reciprocity}
  For $u$ a word and $R\subseteq [0,|u|]$ let $p = (u,R)$, and let
  $\dual{p} = (u, [0,|u|]\setminus R)$. Then, for any word $w$, we have
  $$
  \binom{w}{p} = \sum_{v\in A^*}\coeff{p}{v}\binom{w}{v},\;
  \text{ where }
  \coeff{p}{v} = (-1)^{|v|-|u|}\binom{v}{\dual{p}}.
  $$
\end{theorem}

\begin{proof}
  We shall find it convenient to swap the
  role of $p$ and $\dual{p}$. That is, we shall prove the statement
  \begin{equation}\label{eq:reciprocity}
  \sum_{v\in A^*}(-1)^{|v|-|u|}\binom{v}{p}\binom{w}{v} = \binom{w}{\dual{p}}.
  \end{equation}
  The proof will proceed by induction on the length of $w$. If $w=\eps$,
  the empty word, and $p = (u,R)$, then both sides of
  \eqref{eq:reciprocity} are equal to $\val{u=\eps}$. Assume that $w$ is
  nonempty and write $w = za$ with $z\in A^*$ and $a\in A$.  If
  $p=(\eps,R)$ we have two cases to consider, namely $R=\emptyset$ and
  $R=\{0\}$. Clearly, $\dual{(\eps,\emptyset)} = (\eps,\{0\})$ and vice
  versa. Further, for $a\in A$, we have $\binom{za}{\eps,\{0\}} = 0$,
  $\binom{za}{\eps,\emptyset} = 1$, and
  $$
  \sum_{v\in A^*}(-1)^{|v|}\binom{v}{\eps, R}\binom{za}{v}
  $$
  evaluates to $\binom{za}{\eps} = 1$ if $R=\{0\}$; otherwise, that is, if
  $R=\emptyset$, it evaluates to
  $$\sum_{v\in A^*}(-1)^{|v|}\binom{za}{v} = \sum_{k\geq  0}(-1)^k\binom{|za|}{k} = 0.
  $$
  Let $b\in A$, $u\in A^*$, $k=|u|+1$, and $p=(ub,R)$. In the following
  calculation we will use Lemma~\ref{lemma:pascal} and for convenience
  of notation we will refrain from subtracting $\{k\}$ from $R$ when
  writing the last term of the recursion in that lemma. That is, by
  convention $(u,R) = (u, R\cap [0,|u|])$ so that we disregard any part
  of $R$ that is outside the interval $[0,|u|]$. Using induction, we
  have
  \begin{multline*}
    \sum_{v \in A^*}(-1)^{|v|-k}\binom{v}{p}\binom{za}{v}\\
    \begin{aligned}
      &= \sum_{c\in A}\sum_{v\in A^*}(-1)^{1+|v|-k}\binom{vc}{ub,R}\binom{za}{vc} \\
      &= \sum_{c\in A}\sum_{v\in A^*}(-1)^{1+|v|-k}\binom{vc}{ub,R}
        \left(\binom{z}{vc} + \val{a=c}\binom{z}{v}\right) \\
      &= \binom{z}{ub,R^c} + \sum_{v\in A^*}(-1)^{1+|v|-k}\binom{va}{ub,R}\binom{z}{v} \\
      &= \binom{z}{ub,R^c} + \sum_{v\in A^*}(-1)^{1+|v|-k}\left(
        \val{k\notin R}\binom{v}{ub,R} + \val{a=b}\binom{v}{u,R}\right)\binom{z}{v} \\
      &= \binom{z}{ub,R^c} - \val{k\notin R}\binom{z}{ub,R^c} + \val{a=b}\binom{z}{u,R^c}\\
      &= \val{k\notin R^c}\binom{z}{ub,R^c} + \val{a=b}\binom{z}{u,R^c} \\
      &= \binom{za}{\dual{p}},
    \end{aligned}
  \end{multline*}
  which completes the proof.
\end{proof}

\begin{example}
  If $p = (u,[0,|u|])$ then $\dual{p} = (u,\emptyset)$ and $\binom{w}{p}
  = \delta(w,u)$. So, by the Reciprocity Theorem, $\delta(w,u) =
  \sum_{v\in A^*}(-1)^{|v|-|u|}\binom{v}{u}\binom{w}{v}$, giving us yet
  another proof of \eqref{eq:inverse}.
\end{example}

\begin{example}
  In the introduction we remarked that any function $f:\sym\to\QQ$ can
  be expressed as a unique linear combination of functions
  $\binom{\Cdot}{\sigma}$. Similarly, any function $f : A^* \to \QQ$ can
  be written as a unique, but typically infinite, linear combination of
  functions $\binom{\Cdot}{u}$: $I(A^*)$ acts on the right of the
  function space $\QQ^{A^{\!*}}$ by $(f\ast F)(v) = \sum_{u \leq
    v}f(u)F(u,v)$, and thus $f = \sum_u c(u)\binom{\Cdot}{u}$ is
  equivalent to $f = c \ast P$ and, since $P$ is invertible, $c = f \ast
  P^{-1}$. This is called the \emph{Mahler expansion} of $f$ by Pin and
  Silva~\cite{PiSi2014}.  As an example, let $A=\mathbf{2}$,
  and define the parity function $\xor:A^*\to\mathbf{2}$ by $\xor(w) =
  \val{\,\text{$w$ has an odd number of ones}\,}$. Then
  $$
  \xor(w)= \sum_{k\geq 1}(-2)^{k-1}\binom{w}{1^k}.
  $$
  Indeed, assuming that $w\in A^*$ and $\ell=\binom{w}{1}$ we have
  $$\xor(w)
  = \sum_{k=1}^\ell(-2)^{k-1}\binom{\ell}{k}
  = \frac{1}{2}\bigl(1-(-1)^\ell\bigr) =
  \begin{cases}
    1\, \text{ if $\ell$ is odd}, \\
    0\, \text{ if $\ell$ is even}.
  \end{cases}
  $$
  Similarly, it is easy to prove that if we define the two functions
  $\band,\bor:A^*\to\mathbf{2}$ by $\band(w) = \val{\binom{w}{0}=0}$ and
  $\bor(w) = \val{\binom{w}{1}>0}$, then
  $$\band(w) = \sum_{k\geq 0}(-1)^k\binom{w}{0^k}\quad\text{and}\quad
  \bor(w) = \sum_{k\geq 1}(-1)^{k-1}\binom{w}{1^k}.
  $$
\end{example}

\begin{example}
  Let $A = \{a,b,c\}$ and consider $p = (ac,\{1\})$. Note that $\binom{w}{p}$
  is  the number of occurrences of $ac$ as a \emph{factor} in $w$. We shall
  now use the Reciprocity Theorem to find the Mahler expansion of
  $\binom{\Cdot}{p}$.  We have $\dual{p} = \bigl(ac, \{0,2\}\bigr)$ and
  $\binom{w}{p^*} = \val{\,\text{$w = avc$ for some $v\in A^*$}\,}$.
  Thus,
  $$\binom{w}{ac,\{1\}} = \sum_{v\in A^*}(-1)^{|v|}\binom{w}{avc}.
  $$
\end{example}

\begin{corollary}\label{cor:boolean-reciprocity}
  Let $n$ and $k$ be nonnegative integers. Let $R\subseteq [0,k]$ and
  $R^c = [0,k]\setminus R$. Then the Mahler expansion of the generalized
  binomial coefficient
  $
  \binom{n}{k,R} =
  \#\big\{\,\{s_1,s_2,\dots,s_k\}\subseteq [n]: s_{i+1} = s_i + 1\text{ for } i\in R\,\big\}
  $, with $s_0=0$ and $s_{k+1}=n+1$,
  in $I(\NN)$ is
  $$\binom{n}{k,R} = \sum_{\ell\geq 0} (-1)^{\ell-k}\binom{\ell}{k,R^c}\binom{n}{\ell}.
  $$
\end{corollary}

\begin{proof}
  Let $A=\{a\}$, and let $u,v\in A^*$. Then $u=a^k$ and $v=a^n$ for some
  $k,n\in\NN$. Let $R\subseteq [0,k]$. Then $\binom{v}{u,R} =
  \binom{n}{k,R}$, and the result follows from
  Theorem~\ref{thm:reciprocity}.
\end{proof}

\begin{example}
  Let $n=4$, $k=2$ and $R=\{1\}$. By
  Corollary~\ref{cor:boolean-reciprocity} we have\smallskip
  $$\binom{4}{2,\{1\}} =
  \binom{2}{2,\{0,2\}}\binom{4}{2}
  - \binom{3}{2,\{0,2\}}\binom{4}{3}
  + \binom{4}{2,\{0,2\}}\binom{4}{4}.\smallskip
  $$
  Here, the left-hand side is
  $\#\{\{1,2\},\{2,3\},\{3,4\}\} = 3$, and the right-hand side is
  $\#\{\{1,2\}\}\cdot 6 - \#\{\{1,3\}\}\cdot 4 + \#\{\{1,4\}\}\cdot 1 =
  6-4+1 = 3$.
\end{example}

\bibliographystyle{plain}
\bibliography{ref}

\end{document}